\documentclass[11pt]{article}
\usepackage{amsthm, amsmath, amssymb, amsfonts, url, booktabs, tikz, setspace, fancyhdr, bm}
\usepackage[margin=1in]{geometry}
\usepackage{hyperref, enumerate}
\usepackage[shortlabels]{enumitem}
\usepackage[babel]{microtype}
\usepackage[english]{babel}
\usepackage[capitalise]{cleveref}
\usepackage{comment}
\usepackage{bbm}
\usepackage{csquotes}
\usepackage{tikz-cd}
\usepackage{graphicx}
\usepackage{float}
\usepackage{xcolor}
\usepackage{mathtools}
\usepackage{mathrsfs}
\usepackage{appendix}
\usepackage{amssymb}
\usepackage{enumitem}
\usepackage{listings}
\usetikzlibrary{positioning, arrows.meta, shapes.geometric}

\counterwithin{figure}{section}

\newtheorem{theorem}{Theorem}[section]
\newtheorem{prop}[theorem]{Proposition}
\newtheorem{lemma}[theorem]{Lemma}
\newtheorem{cor}[theorem]{Corollary}

\usetikzlibrary{decorations.pathmorphing}

\theoremstyle{definition}
\newtheorem{defn}[theorem]{Definition}
\newtheorem*{defn-non}{Definition}

\newtheorem{probl}[theorem]{Problem}

\newlist{Case}{enumerate}{2}
\setlist[Case, 1]{%
    label           =   {\bfseries Case \arabic*.},
    labelindent=1em ,labelwidth=1.3cm, labelsep*=1em, leftmargin =!
}
\setlist[Case, 2]{%
    label           =   {\bfseries Subcase \arabic{Casei}.\arabic*.},
    labelindent=-1em ,labelwidth=1.3cm, labelsep*=1em, leftmargin =!
}

\usepackage{todonotes}

\newcommand{\ex}{\mathrm{ex}}

\newcommand{\Gr}{\operatorname{Gr}}

\newcommand{\w}{{\scriptscriptstyle\Lambda}}

\DeclareMathOperator{\Alt}{Alt}

\DeclareMathOperator{\Hom}{Hom}

\DeclareMathOperator{\spa}{span}
\DeclareMathOperator{\ch}{char}
\DeclareMathOperator{\FP}{FP}
\DeclareMathOperator{\GQ}{GQ}
\DeclareMathOperator{\T}{T}

\title{Tur\'{a}n Problems for Multilinear Maps} 
\author{
Qiyuan Chen\thanks{State Key Laboratory of Mathematical Sciences, Academy of Mathematics and Systems Science, Chinese Academy of Sciences. Emails: chenqiyuan@amss.ac.cn, keyk@amss.ac.cn}
\and
Zixiang Xu\thanks{School of Mathematical Sciences, Zhejiang University, Hangzhou, China. Email: zixiangxu@zju.edu.cn.}
\and
Ke Ye\footnotemark[1]
}

\begin{document}
\date{}
\maketitle

\begin{abstract}
We study Tur\'{a}n-type extremal problems for alternating and unrestricted multilinear maps. For alternating order-$d$ multilinear maps $T: (\mathbb F^n)^d\to \mathbb{F}^m$, we determine, over algebraically closed fields of arbitrary characteristic, the largest $k$ such that every $T$ vanishes identically on $\mathbb{V}^d$ for some $k$-dimensional subspace $\mathbb{V}$. This extends the bilinear formula of Buhler, Gupta, and Harris [J. Algebra, 1987] to arbitrary order and resolves a question of Qiao [Discrete Anal., 2023]. We also solve the analogous problem for arbitrary, not necessarily alternating, multilinear maps by determining the largest $k$ such that every $T$ vanishes on $\mathbb{V}_1\times\cdots\times \mathbb{V}_d$ for some $k$-dimensional subspaces $\mathbb{V}_1,\dots,\mathbb{V}_d$. These results yield exact values, over algebraically closed fields, of the Feldman--Propp number [Adv. Math., 1992], the Tur\'{a}n number [Discrete Anal., 2023], and the Gow--Quinlan number [Linear Multilinear Algebra, 2006] associated with alternating multilinear maps. Finally, motivated by the Erdős box problem, we give a purely algebraic derivation of the Conlon--Pohoata--Zakharov lower bound [Discrete Anal., 2021] by combining analytic and partition rank estimates with an incidence count. In the relevant parameter range, we further show that every multilinear map defined over a finite field has many isotropic tuples of $2$-dimensional subspaces over extensions of sufficiently divisible degree. This rules out the natural route to improving the Conlon--Pohoata--Zakharov exponent by selecting multilinear maps with substantially fewer bad isotropic configurations.
\end{abstract}

\section{Introduction}
Multilinear maps and their isotropic subspaces have been extensively studied in the context of group theory \cite{Alperin65,Baer38,buhler1987isotropic,eberhard2025probabilistic,GQ17,RV24}, arithmetic combinatorics \cite{chen2025isotropy,DS10,FP92,qiao2020tur},  graph theory~\cite{Bukh15,Bukh24,chen2025graph,CPZ20},  algebraic geometry \cite{BR21,Catanese91,Karpenko16,Tevelev01} and computer science \cite{BCGQS21,BMW17,GQ17,Sun23}.  This paper focuses on Tur\'{a}n problems for (alternating) multilinear maps,  concerning the largest $k$ such that every such map admits a $k$-dimensional isotropic subspace.  
\subsection{Main Results} 
An order-$d$ multilinear map $T: \mathbb{F}^n \times \cdots \times \mathbb{F}^n \to \mathbb{F}^m$ is said to be \emph{alternating} if $T(u_1,\dots, u_d) = 0$ whenever $u_i = u_j$ for some distinct $1 \le i,j \le d$. We denote by $\Alt^d(\mathbb{F}^n, \mathbb{F}^m)$ the space of all order-$d$ alternating multilinear maps from $(\mathbb{F}^n)^d$ to $\mathbb{F}^m$. Here and hereafter, we write $X^d$ for the Cartesian product of $d$ copies of a set $X$. Moreover,  we use $\Gr(k, \mathbb{F}^n)$ to denote the \emph{Grassmannian} of $k$-dimensional subspaces in $\mathbb{F}^n$. Throughout this paper, we write $\mathbb{N}$ for the set of positive integers and $\mathbb{N}_0$ for the set of non-negative integers. Given $T \in \Alt^d(\mathbb{F}^n,  \mathbb{F}^m)$,  a subspace $\mathbb{V} \in \Gr(k,\mathbb{F}^n)$ is \emph{isotropic} if $T\big\vert_{\mathbb{V}} = 0$, that is,  $T(v_1,\dots,  v_d) = 0$ for all $v_1,\dots,  v_d \in \mathbb{V}$.  The following Tur\'{a}n problem for $\Alt^d(\mathbb{F}^n,  \mathbb{F}^m)$ was posed by Qiao in \cite[Section~5.7]{qiao2020tur}.
\begin{probl}[Tur\'{a}n problem for $\Alt^d(\mathbb{F}^n,\mathbb{F}^m)$]\label{probl:TuranAlt}
Given a field $\mathbb{F}$ and positive integers $n,d,m$,  find the largest number $k \coloneqq \alpha_{\w}(\mathbb{F},n,d,m)$ such that every $T\in \Alt^d(\mathbb{F}^n,\mathbb{F}^m)$ has a $k$-dimensional isotropic subspace.  
\end{probl} 
We refer interested readers to \cite[Section~5]{qiao2020tur} for a detailed discussion of the analogy between the classical Tur\'{a}n problem for hypergraphs and Problem~\ref{probl:TuranAlt}.  For $d = 2$,  Problem~\ref{probl:TuranAlt} was first studied by Ol'\v{s}anski\u{i} \cite{sanski78} over finite fields,  and later by Buhler, Gupta and Harris \cite{buhler1987isotropic} for arbitrary fields.  By the main theorem of \cite{buhler1987isotropic}, \begin{equation}\label{eq:BGH}
\alpha_{\w}(\mathbb{F},n,2,m) = \Big\lfloor \frac{2n + m}{m+2} \Big\rfloor
\end{equation}
holds for $m \ge 2$ and algebraically closed fields $\mathbb{F}$ with $\ch (\mathbb{F}) \ne 2$.  As commented in \cite[Section~5.7]{qiao2020tur},  generalizing the intersection-theoretic method in \cite{buhler1987isotropic,HT84} to calculate $\alpha_{\w}(\mathbb{F},n,d,m)$ for $d \ge 3$ poses substantial challenges. Our first main result generalizes \eqref{eq:BGH} to $d \ge 3$,  solving Problem~\ref{probl:TuranAlt} for algebraically closed fields of arbitrary characteristic.
\begin{theorem}\label{main theorem-1}
Let $\mathbb{F}$ be an algebraically closed field and $n,d,m$ be positive integers.  
\begin{enumerate}[(a)]
\item\label{main theorem-1:item1} If $m \ge 2$, then $\alpha_{\w}(\mathbb{F},n,d,m) = \max \big\lbrace 
s \in \mathbb{N}_0: s(n-s)\ge m\binom{s}{d}
\big\rbrace$.
\item\label{main theorem-1:item2}If $m =  1$ , then
\end{enumerate}
\[
\alpha_{\w}(\mathbb{F},n,d,1) = 
\begin{cases}
\lceil \frac{n}{2} \rceil \quad &\text{if~} d = 2, \\
4 \quad &\text{if~} (d,n) = (3,7),  \\
n - 2 \quad &\text{if~} d = n-2 \text{~is even},  \\
\max \big\lbrace 
s \in \mathbb{N}_0: s(n-s)\ge \binom{s}{d}
\big\rbrace \quad &\text{otherwise}. 
\end{cases}
\]
\end{theorem}
Our proof of Theorem~\ref{main theorem-1}\ref{main theorem-1:item2} relies on a result of Tevelev \cite{Tevelev01} concerning the isotropy index of a generic element in $\Alt^d(\mathbb{F}^n,\mathbb{F})$ when $\ch (\mathbb{F}) = 0$. In particular, the characteristic-zero case of Theorem~\ref{main theorem-1}\ref{main theorem-1:item2} follows immediately from Tevelev's result together with Corollary~\ref{coro-irreducible and closed}. It is worth emphasizing that Tevelev's proof uses the Borel--Weil--Bott theorem to establish a necessary condition for a generic section of a line bundle to vanish. However, this argument does not extend directly to the positive-characteristic case. Furthermore, our proof of Theorem~\ref{main theorem-1}\ref{main theorem-1:item2} does not extend to \(\Alt^d(\mathbb{F}^n,\mathbb{F}^m)\) with $m\ge2$, since in that setting one must study the common zero locus of several global sections of a line bundle rather than the zero locus of a single section. Therefore, \ref{main theorem-1:item1} and \ref{main theorem-1:item2} require fundamentally different arguments, and we prove them separately in Theorems~\ref{main theorem-1a} and \ref{main theorem-1b}, respectively.

Let $\Hom^d(\mathbb{F}^{n},  \mathbb{F}^m)$ be the space of order-$d$ multilinear maps from $(\mathbb{F}^{n})^d$ to $\mathbb{F}^m$.  For $T\in \Hom^d(\mathbb{F}^{n},  \mathbb{F}^m)$ and $(\mathbb{V}_1,\dots,  \mathbb{V}_d) \in \Gr(k,  \mathbb{F}^n)^d$,  we write $T|_{\mathbb{V}_1 \times \cdots \times  \mathbb{V}_d} = 0$ if $T(v_1,\dots,  v_d) = 0$ for any $(v_1,\dots,  v_d) \in \mathbb{V}_1 \times \cdots \times \mathbb{V}_d$. In this case, we say that $T$ has an \emph{isotropic $d$-tuple} of $k$-dimensional subspaces. The Tur\'{a}n problem for $\Hom^d(\mathbb{F}^{n},  \mathbb{F}^m)$ is stated as follows. 
\begin{probl}[Tur\'{a}n problem for $\Hom^d(\mathbb{F}^{n},  \mathbb{F}^m)$]\label{probl:Turan}
Given a field $\mathbb{F}$ and $n,d,m \in \mathbb{N}$, find the largest number $k \coloneqq \alpha(\mathbb{F}, n, d,m)$ such that for each $T\in \Hom^d(\mathbb{F}^{n}, \mathbb{F}^m)$,  there exists $(\mathbb{V}_1,\dots,  \mathbb{V}_d) \in \Gr(k,  \mathbb{F}^n)^d$ such that $T|_{\mathbb{V}_1 \times \cdots \times  \mathbb{V}_d} = 0$.
\end{probl}
Problem~\ref{probl:Turan} is seemingly analogous to Problem~\ref{probl:TuranAlt},  yet substantially different: alternating multilinear maps are highly structured,  whereas general ones are not,  and such structure typically yields desirable properties.  For illustration,  we recall from \cite{chen2025isotropy} that the \emph{isotropy indices} of $T\in \Hom^d(\mathbb{F}^{n},  \mathbb{F}^m)$ and $S \in \Alt^d(\mathbb{F}^n,  \mathbb{F}^m)$ are defined as
\begin{align*}
\alpha (T) &\coloneqq \max \big\lbrace
s \in \mathbb{N}_0: T|_{\mathbb{V}_1 \times \cdots \times \mathbb{V}_d} = 0 \text{~for some~} (\mathbb{V}_1,\dots,  \mathbb{V}_d) \in \Gr(s,  \mathbb{F}^n)^d
\big\rbrace, \\ 
\alpha_{\w} (S) &\coloneqq \max \big\lbrace
s \in \mathbb{N}_0: S|_{\mathbb{V}} = 0 \text{~for some~} \mathbb{V} \in \Gr(s,  \mathbb{F}^{n})
\big\rbrace.  
\end{align*}
Clearly, $\alpha(\mathbb{F}, n, d, m) = \min_T \alpha(T)$ and $\alpha_{\w}(\mathbb{F},n,d,m) = \min_S \alpha_{\w} (S)$,  where $T$ and $S$ range over $\Hom^d(\mathbb{F}^{n},  \mathbb{F}^m)$ and $\Alt^d(\mathbb{F}^n,\mathbb{F}^m)$ respectively. By \cite[Remark~3.7]{chen2025isotropy},  $\alpha_{\w}$ is additive under direct sums,  while $\alpha$ is only superadditive.  Moreover,  the intersection-theoretic method of \cite{buhler1987isotropic} for \eqref{eq:BGH} relies on computations in the Chow ring $A^\ast(\Gr(k,\mathbb{F}^n))$.  By contrast,  the same technique for $\alpha(\mathbb{F}, n_1,  n_2,  m)$ requires computations in the more complicated ring $A^\ast(\Gr(k,\mathbb{F}^{n_1})) \otimes A^\ast(\Gr(k,\mathbb{F}^{n_2}))$. In Subsection~\ref{subsec:MultilinearTuran}, we establish the formula for $\alpha(\mathbb{F},n,d,m)$ when $\mathbb{F}$ is algebraically closed.

\begin{theorem}\label{main theorem-2}
Let $\mathbb{F}$ be an algebraically closed field. Given $n,d,m \in \mathbb{N}$, we have
\begin{enumerate}[(a)]
    \item\label{main theorem-2:item1} If $m \ge 2$, then $\alpha(\mathbb{F},n,d,m) = \max\big\{s\in\mathbb{N}_0: ds(n-s)\ge m s^d\big\}$.
    \item\label{main theorem-2:item2} If $m = 1$, then 
    \[
    \alpha(\mathbb{F},n,d,1) = \begin{cases}
    n - 1 \quad &\text{if~} d = 1, \\
    \lfloor \frac{n}{2} \rfloor \quad &\text{if~} d = 2, \\
    \max\big\{s\in\mathbb{N}_0: ds(n-s)\ge s^d\big\} \quad &\text{otherwise}.
    \end{cases}
    \]
\end{enumerate}
\end{theorem}

Theorem~\ref{main theorem-2} is closely connected to the classical problem of determining the maximal dimension of a completely entangled subspace \cite{bennett99,demianowicz21,parthasarathy04,Wallach02}. More precisely, a subspace of $(\mathbb{C}^n)^{\otimes d}$ is called \emph{completely entangled} if it contains no nonzero decomposable tensor. The natural isomorphism $\Hom^d(\mathbb{C}^n,\mathbb{C}^m) \cong \Hom((\mathbb {C}^n)^{\otimes d}, \mathbb{C}^m)$ implies that every $T\in \Hom^d(\mathbb{C}^n,\mathbb{C}^m)$ corresponds to a linear map $\widetilde T:(\mathbb C^n)^{\otimes d}\to\mathbb C^m$, and $\alpha(T)=0$ if and only if $\ker (\widetilde T)$ is completely entangled. Consequently, if $(\mathbb{C}^n)^{\otimes d}$ has a completely entangled subspace of dimension $N$, then $\alpha(\mathbb{C},n,d,n^d-N) = 0$ and Theorem~\ref{main theorem-2} recovers the classical bound $N \le n^d-d(n-1)-1$ established in \cite{parthasarathy04,Wallach02}.

\subsection{Applications}
We briefly discuss some applications of Theorems~\ref{main theorem-1} and \ref{main theorem-2}.  
\subsubsection{Feldman-Propp Number} The first application is concerned with the \emph{Feldman-Propp number} \cite{FP92,qiao2020tur}  defined as 
\[
\FP(\mathbb{F},d,m,k) \coloneqq \min \{ n \in \mathbb{N}: \alpha_{\w}(T) \ge k \text{~for all~}T \in \Alt^d(\mathbb{F}^n,\mathbb{F}^m) \}.
\]
In Section~5 of \cite{FP92},  a lower bound for $\FP(\mathbb{F},d,m,k)$ was established: 
\begin{equation}\label{eq:FP}
\FP(\mathbb{F},d,m,k)  \ge \Big\lceil \frac{m}{k} \binom{k}{d} \Big\rceil + k,
\end{equation}
and the authors stated without proof that there is an upper bound that is not far from the lower bound when $\mathbb{F}$ is algebraically closed with $\ch (\mathbb{F}) = 0$.  By definition,  we have $\alpha_{\w}(\mathbb{F},  n,d,m) \ge k$ if and only if $\FP(\mathbb{F},d,m,k) \le n$.  As a direct consequence of Theorem~\ref{main theorem-1},  we obtain that the equality in \eqref{eq:FP} almost always holds when $\mathbb{F}$ is algebraically closed.
\begin{cor}
If $\mathbb{F}$ is algebraically closed, then
\[
\FP(\mathbb{F},d,m,k)  = 
\begin{cases}
2k - 1 \quad &\text{if~} d = 2 \text{~and~} m = 1, \\
8 \quad &\text{if~} (d,m,k) = (3,1,5), \\
d+3\quad &\text{if~} m = 1, k = d+1 \text{~and~} d\ge 4 \text{~is even}, \\
\lceil \frac{m}{k} \binom{k}{d} \rceil + k \quad &\text{otherwise}.
\end{cases}
\]
\end{cor}

\subsubsection{Tur\'{a}n Number and Gow-Quinlan Number}
The \emph{Tur\'{a}n number} for alternating multilinear maps is defined in  \cite[Section~5.4]{qiao2020tur} as a direct generalization of the classical Tur\'{a}n number for hypergraphs.  It is defined as 
\[
\T(\mathbb{F},n,d,k) \coloneqq \min \lbrace
r \in \mathbb{N}: \alpha_{\w}(\mathbb{F},n,d,r) \le k
\rbrace.  
\]
Since $n \ge \alpha_{\w}(\mathbb{F},n,d,r) \ge \min\{ d-1,  n\}$,  we may consider the \emph{Gow-Quinlan number},  given by
\[
\GQ(\mathbb{F},n,d) = \T(\mathbb{F},n,d,d-1) \coloneqq \min
\lbrace
r \in \mathbb{N}:  \alpha_{\w}(\mathbb{F},n,d,r) = d-1
\rbrace.
\]
The Gow-Quinlan number was first investigated in \cite{GQ06} for $d = 2$.  From \eqref{eq:BGH},  Gow and Quinlan observed  that $\GQ(\mathbb{F},n,2) =  2n - 3$ when $\mathbb{F}$ is algebraically closed and $\ch (\mathbb{F}) \ne 2$. As an immediate consequence of Theorem~\ref{main theorem-1},  we determine the values of both the Tur\'{a}n number and the Gow-Quinlan number over algebraically closed fields for arbitrary $d \ge 2$.
\begin{cor}
Let $\mathbb{F}$ be an algebraically closed field.  For any positive integers $n,k$ and $d \ge 2$ such that $2 \le d \le n$ and $d - 1 \le k \le n$, we have 
\[
\T(\mathbb{F},n,d,k) = \max \bigg\lbrace 0, \bigg\lfloor  
\frac{(k+1)(n - k - 1)}{\binom{k+1}{d}} \bigg\rfloor \bigg\rbrace + 1
\]
with three exceptional cases: $d = 2$ and $\lceil n / 2 \rceil \le k$; $(d,n) = (3,7)$ and $4 \le k$; $d = n-2$,  $n$ is even and $n - 2 \le k$, in which $\T(\mathbb{F},n,d,k) = 1$. In particular,  for any algebraically closed field $\mathbb{F}$ and positive integers $n$,  $d$,  it holds that $\GQ(\mathbb{F},n,d) = \max\{0,  d(n - d )\} + 1$.
\end{cor}

\subsubsection{Barrier for the Multilinear Method in the Erd\H{o}s Box Problem}\label{subsubsec:Erdos}
The famous Erd\H{o}s box problem \cite{1964Israel} asks for the value of the extremal number 
\[
\ex_d(N,  \mathcal{K}^{(d)}_{2,\dots,  2}) = \max \big\lbrace |E(\mathcal{G})|: \mathcal{G} \text{~is an $N$-vertex, $d$-uniform hypergraph not containing $\mathcal{K}^{(d)}_{2,\dots,  2}$} \big\rbrace,  
\]
where $\mathcal{K}^{(d)}_{2,\dots,  2}$ is a complete $d$-partite $d$-uniform hypergraph with each part of size two.  By a multilinear method, Conlon,  Pohoata and Zakharov \cite{CPZ20} proved that 
\begin{equation}\label{eq:CPZ}
\ex_d(N,  \mathcal{K}^{(d)}_{2,\dots,  2}) = \Omega (N^{d - \frac{m}{n}})
\end{equation} 
for any positive integers $m,n$ such that {$d(n - 1) < (2^d - 1)m$, which refined earlier constructions of Gunderson, R\"{o}dl, and Sidorenko based on random hyperplanes~\cite{1999JCTA}. The last application is a barrier result for this method,  which we first reformulate in purely algebraic terms.  

Let $d \ge 2$ be an integer and let $\mathscr{H}_d$ be the set of $d$-partite $d$-uniform hypergraphs. According to Bukh's random algebraic construction \cite{Bukh15,  Bukh24},  there is a map
\begin{equation}\label{eq:Erdox}
\mathcal{G}: \bigsqcup_{\substack{n,m \in \mathbb{N},\\ q \text{: prime power}}} \Hom^d (\mathbb{F}_q^{n+1},  \mathbb{F}_q^m) \to \mathscr{H}_d,\quad \mathcal{G}(T)  = (V_T,  E_T)
\end{equation}
where $V_T$ is the disjoint union of $d$ copies of $\mathbb{P}^n(\mathbb{F}_q)$,  and $E_T$ consists of all $([v_1],\dots,  [v_d]) \in (\mathbb{P}^n(\mathbb{F}_q) )^d$ such that $T(v_1,\dots,  v_d) = 0$. 

A key observation is that hyperedges of each $\mathcal{K}^{(d)}_{2,\dots,  2}$ are contained in some isotropic $d$-tuple of two-dimensional subspaces. Thus, deleting all hyperedges corresponding to isotropic $d$-tuples of two-dimensional subspaces of $T$ yields a hypergraph $\mathcal{G}'(T)$ without $\mathcal{K}^{(d)}_{2,\dots, 2}$.  A lower bound for $\ex_d(N,  \mathcal{K}^{(d)}_{2,\dots,  2})$ follows from an upper bound for $|D_T|$, where $D_T \coloneqq\{( \mathbb{V}_{1},\dots,\mathbb{V}_{d})\in \Gr(2,  \mathbb{F}_q^{n+1})^{d}: T|_{\mathbb{V}_{1}\times \cdots \times \mathbb{V}_{d}}=0\}$.  The lower bound~\eqref{eq:CPZ} was proved in~\cite{CPZ20} via the random multilinear method:
one samples a random multilinear map $T$ from an appropriate finite model, estimates the expectation $\mathbb{E}|D_T|$,
and then deletes the hyperedges coming from bad configurations to obtain a $K^{(d)}_{2,\ldots,2}$-free hypergraph with many edges. 

In Appendix~\ref{sec:append}, we give a purely algebraic proof of~\eqref{eq:CPZ} using analytic rank and partition rank, notions originating in additive combinatorics. In particular, our argument yields a deterministic upper bound on $|D_{T}|$ for each admissible multilinear map $T$, rather than a upper bound on $\mathbb{E}|D_{T}|$, which may be of independent interest in Tur\'{a}n-type problems. Moreover, we show that there is an obstruction to improving the exponent in~\eqref{eq:CPZ} within this multilinear framework.

\begin{cor}\label{cor:Erdox}
There is a function $C_1: \mathbb{N} \times \mathbb{N} \times \mathbb{N} \to \mathbb{N}$ with the following property. For any prime power $q$ and positive integers $n,d, m,r$ such that $d\ge 3$, $2^{d-1}m \le d(n-1) < (2^d - 1)m$ and $r$ is divisible by $C_1(n,d,m)$, there is no $T \in \Hom^d(\mathbb{F}_{q}^{n+1}, \mathbb{F}_{q}^m)$ such that $\vert D_T(\mathbb{F}_{q^r})\vert \le 1/2 q^{2r(d(n-1)-m2^{d-1})}$.
\end{cor}

An important implication of Corollary~\ref{cor:Erdox} is that any direct improvement of \eqref{eq:CPZ} by the aforementioned multilinear method, if not impossible, can only occur over small fields. The proof of Corollary~\ref{cor:Erdox} makes use of constructions from the proof of Theorem~\ref{main theorem-2}. Therefore, we defer it to Appendix~\ref{sec:append}.

\subsection{Organization}
The rest of this paper is organized as follows. In Section~\ref{sec:prelim},  we record some basic lemmas in algebraic geometry, for ease of reference.  Theorem~\ref{main theorem-1} is proved in Subsection~\ref{subsec:alternatingTuran},  and Subsection~\ref{subsec:MultilinearTuran} is devoted to the proof of Theorem~\ref{main theorem-2}. In Appendix~\ref{sec:append},  we provide an algebraic proof of \eqref{eq:CPZ} and the complete proof of Corollary~\ref{cor:Erdox}.

\section{Basic Facts}\label{sec:prelim}
We collect some basic facts from algebraic geometry below.   
\begin{lemma}{\normalfont\cite[Theorem~11.14]{harris2013algebraic}}\label{thm:irreducible_fibre}
Let $\pi:X\to Y$ be a surjective morphism of varieties. If $Y$ is irreducible and for each $y\in Y$, the fibre $\pi^{-1}(y)$ is irreducible and of the same dimension,  then $X$ is irreducible. 
\end{lemma}
\begin{lemma}[Fiber dimension formula]\cite[Theorem~11.12]{harris2013algebraic} \& \cite[Proposition 10.6.1(iii)]{grothendieck1966elements} \label{dimension of fibre}
    Let $\pi:X\to Y$ be a morphism of varieties.
    \begin{enumerate}[(a)]
        \item If $X$ is an irreducible quasi-projective variety, $Y$ is an irreducible projective variety and $\pi$ is dominant, then $\dim \pi^{-1}(y) = \dim X -\dim Y $  for generic $y\in Y$.
        \item If for any $y\in Y$, $\dim \pi^{-1}(y) \ge r$(resp. $\dim \pi^{-1}(y) \le r$), then $\dim X \ge r+\dim Y $(resp.  $\dim X \le r+\dim Y $)  
    \end{enumerate}
\end{lemma}
\begin{lemma}{\normalfont\cite[Corollary~11.13]{harris2013algebraic}}\label{upper-semicontinuous}
Let $X,Y$ be projective varieties, and $\pi:X\to Y$ be a regular morphism.  Then the function $\lambda: Y \to \mathbb{N}_0 \cup \{-\infty\}$ defined by  $\lambda(q)=\dim \pi^{-1}(q)$ is upper-semicontinuous. 
\end{lemma} 

\begin{lemma}[Valuative criterion]\cite[Theorem~4.7]{hartshorne2013algebraic}\label{lem:properness}
Let $f \colon X \to Y$ be a morphism of schemes. Suppose that $X$ is Noetherian and $f$ is of finite type. Then $f$ is proper if and only if for every valuation ring $R$ with the field of frations $\mathbb{K}$, and for every commutative diagram
\[
\begin{tikzcd}[row sep=1em, column sep=1em]
	{\mathrm{Spec} \mathbb{K}} &&&& X \\
	\\
	\\
	{\mathrm{Spec} R} &&&& Y
	\arrow[from=1-1, to=1-5]
	\arrow["i"', from=1-1, to=4-1]
	\arrow["f", from=1-5, to=4-5]
	\arrow[dashed, from=4-1, to=1-5]
	\arrow[from=4-1, to=4-5]
\end{tikzcd}
\]
in which the left vertical morphism is induced by the inclusion
$R \subseteq \mathbb{K}$, there exists a unique morphism
$\operatorname{Spec} R \to X$ (the dashed arrow) making the whole diagram
commute.
\end{lemma}

Given a perfect field $\mathbb{F}$ of positive characteristic, we denote its \emph{ring of Witt vectors} by $W(\mathbb{F})$ \cite[Chapter~II \S~6]{serre1979local}.
\begin{lemma}\cite[Chapter~II Theorem~8]{serre1979local}\label{property of Witt vector}
If $\mathbb{F}$ is a perfect field of positive characteristic, then $W(\mathbb{F})$ is a complete discrete valuation ring with residue field $\mathbb{F}$. Moreover, the field of fractions of $W(\mathbb{F})$ has characteristic zero.
\end{lemma}

We recall from \cite[Section~1]{Kedlaya22} that a valued field $(\mathbb{F},\nu)$ is \emph{complete} if $\mathbb{F}$ is complete with respect to the topology induced by $\nu$. We say that $\nu$ is \emph{non-archemidean} if $\nu$ is ultrametric.
\begin{lemma}\cite[Theorem~1.4.9 and Lemma~3.1.1]{Kedlaya22}\label{property of local field}
Suppose that $(\mathbb{F},\nu)$ is a complete non-archemidean valued field and $\mathbb{L}$ is a finite extension of $\mathbb{F}$. Then $\nu$ extends uniquely to an absolute value $\omega$ on $\mathbb{L}$, and $(\mathbb{L},\omega)$ is complete. Moreover, $\kappa(\mathbb{L})$ is a finite extension of $\kappa(\mathbb{F})$.
\end{lemma}

\section{Proofs of Main Theorems}
Let $\mathbb{F}$ be a field.  We write $\Hom^d(\mathbb{F}^{n},  \mathbb{F}^m)$ for the space of all multilinear maps from $(\mathbb{F}^n)^d$ to $\mathbb{F}^{m}$.  Moreover, we denote by $\Alt^d(\mathbb{F}^{n},  \mathbb{F}^{m})$ the subspace of $\Hom^d(\mathbb{F}^{n},  \mathbb{F}^m)$ consisting of all alternating multilinear maps.  This section is devoted to proving Theorems~\ref{main theorem-1} and~\ref{main theorem-2}. For ease of reference,  we recall the definition of isotropy indices below.
\begin{defn}[Isotropy index]
For $n,  d,  m \in \mathbb{N}$, the isotropy index of $\Alt^d(\mathbb{F}^{n},  \mathbb{F}^{m})$ is defined by 
\[
\alpha_{\w}(\mathbb{F},n,d, m) \coloneqq \max \big\lbrace 
k \in \mathbb{N}_0: \alpha_{\w}(T) \ge k\ \textup{for\ any\ } T \in \Alt^d(\mathbb{F}^{n},  \mathbb{F}^{m}) 
\big\rbrace
\]
Given $n, m \in \mathbb{N}$, the isotropy index of $\Hom^d(\mathbb{F}^{n},  \mathbb{F}^m))$ is defined by 
\[
\alpha(\mathbb{F},  n, d,m) \coloneqq
\max \big\lbrace 
k \in \mathbb{N}_0: \alpha(T) \ge k\ \textup{for\ any\ } T \in \Hom^d(\mathbb{F}^{n},  \mathbb{F}^m)
\big\rbrace.
\]
\end{defn}
Equivalently,  we have $\alpha_{\w}(\mathbb{F},n,d, m) = \min \big\lbrace \alpha_{\w}(T): T \in \Alt^d(\mathbb{F}^{n},  \mathbb{F}^{m}) \big\rbrace$ and $\alpha(\mathbb{F},  n, d,m) = 
 \min \big\lbrace \alpha(T): T \in \Hom^d(\mathbb{F}^{n},  \mathbb{F}^m) \big\rbrace$.  It is clear that 
\[
\min \{d-1,n\}  \le \alpha_{\w}(\mathbb{F},n,d,m)\le n ,  \quad 0\le \alpha(\mathbb{F},  n, d,m) \le n.
\]
 
\subsection{Tur\'{a}n Problem for $\Alt^d(\mathbb{F}^{n},  \mathbb{F}^{m})$}\label{subsec:alternatingTuran}
We first prove Theorem~\ref{main theorem-1}. Let $\mathbb{F}$ be a field and let $d,n,m,k$ be positive integers.  We denote by $\mathbb{P} ( \Alt^d(\mathbb{F}^n,  \mathbb{F}^m) )$ the projectivization of the vector space $\Alt^d(\mathbb{F}^n,  \mathbb{F}^m)$.  Given $T \in \Alt^d(\mathbb{F}^n,  \mathbb{F}^m)$,  we write $T|_{\mathbb{V}} = 0$ if $\mathbb{V} \subseteq \mathbb{F}^n$ is an isotropic subspace of $T$.  We define the incidence set for $k$-dimensional isotropic subspaces as 
\begin{equation}\label{eq:incidence}
I_1 \coloneqq \big\lbrace
(\mathbb{V}, [T]) \in \Gr(k,  \mathbb{F}^n) \times \mathbb{P} ( \Alt^d(\mathbb{F}^n,  \mathbb{F}^m) ): T|_{\mathbb{V}} = 0
\big\rbrace.
\end{equation}
\begin{lemma}\label{lem-irreducible incidence}
For integers $d, n,k$ satisfying $d \le n$ and $0 \le k \le n$, the incidence set $I_1$ is an irreducible closed subvariety of $\Gr(k,  \mathbb{F}^n) \times \mathbb{P} ( \Alt^d(\mathbb{F}^n,  \mathbb{F}^m) )$,  and $\dim I_1 = (n-k)k+m \binom{n}{d}- m \binom{k}{d} -1$.
\end{lemma}
\begin{proof}
We consider 
\[
Z \coloneqq
\big\lbrace
(\mathbb{V},[T],  [v_{1}],\dots, [v_{d}])\in \Gr(k,\mathbb{F}^n)\times \mathbb{P}(\Alt^{d}(\mathbb{F}^{n},\mathbb{F}^{m}))\times(\mathbb{P}^{n-1})^{d}: v_1,\dots,  v_d \in \mathbb{V},\; T(v_{1},\dots,v_{d})=0
\big\rbrace,
\]
where $[T]$ and $[v]$ are elements in $\mathbb{P}(\Alt^{d}(\mathbb{F}^{n},\mathbb{F}^{m}))$ and $\mathbb{P}^{n-1}$ represented by $T \in \Alt^{d}(\mathbb{F}^{n},\mathbb{F}^{m}) \setminus \{0\}$ and $v \in \mathbb{F}^n \setminus \{0\}$,  respectively. By definition,  $Z$ is a projective variety.  

Let $\pi: Z \to \Gr(k,\mathbb{F}^n) \times \mathbb{P}(\Alt^{d}(\mathbb{F}^{n},\mathbb{F}^{m}))$ be the natural projection to the first two components. Denote 
\[
X \coloneqq \big\lbrace
(\mathbb{V},[T]) \in \Gr(k,\mathbb{F}^n)\times \mathbb{P}(\Alt^{d}(\mathbb{F}^{n},\mathbb{F}^{m})):  \dim(\pi^{-1}(\mathbb{V}, [T])) \ge (k-1)d
\big\rbrace.
\] 
Since each $(\mathbb{V},T,  [v_1],\dots,  [v_d]) \in \pi^{-1}(\mathbb{V},[T])$ must satisfy $T|_{\mathbb{W}} = 0$ where $\mathbb{W} \coloneqq \spa_{\mathbb{F}}\lbrace v_1,\dots,  v_d \rbrace \subseteq \mathbb{V}$, it holds that $\dim \pi^{-1}(\mathbb{V},[T]) \le (k-1)d$.  Thus,  we must have $T|_{\mathbb{V}} = 0$ for each $(\mathbb{V},[T]) \in X$.  In particular, this implies that $X \subseteq I_1$.  Conversely,  if $(\mathbb{V}, [T]) \in I_1$,  then 
\[
\pi^{-1} (\mathbb{V},[T]) \supseteq \big\lbrace
(\mathbb{V},[T],  [v_1],\dots,  [v_d]) \in \Gr(k,\mathbb{F}^n)\times \mathbb{P}(\Alt^{d}(\mathbb{F}^{n},\mathbb{F}^{m}))\times(\mathbb{P}^{n-1})^{d}: T|_{\mathbb{V}} = 0,\;v_1,\dots,  v_d\in \mathbb{V} 
\big\rbrace,
\]
from which we conclude that $\dim \pi^{-1}(\mathbb{V},  [T]) \ge (k - 1)d$ and $(\mathbb{V},  [T]) \in X$. Therefore, we have $I_1 = X$. By Lemma~\ref{upper-semicontinuous}, $I_1$ is closed.

Let $\pi_1: I_1 \to \Gr(k,\mathbb{F}^n)$ be the natural projection to the first component.  We observe that $I_1$ is a projective bundle on $\Gr(k,\mathbb{F}^n)$ via $\pi_1$,  the fiber of which has dimension $m\big( \binom{n}{d} - \binom{k}{d}\big) - 1$.  Hence $I_1$ is irreducible of the desired dimension by Lemmas~\ref{thm:irreducible_fibre} and \ref{dimension of fibre}.
\end{proof}
Next,  we consider 
\begin{equation}\label{eq:V}
V \coloneqq \big\lbrace
[T] \in \mathbb{P}(\Alt^{d}(\mathbb{F}^{n},\mathbb{F}^{m})): \alpha_{\w}(T)\ge k
\big\rbrace.
\end{equation}
Let $\pi_2: I_1 \to \mathbb{P}(\Alt^{d}(\mathbb{F}^{n},\mathbb{F}^{m}))$ be the natural projection to the second component.  Then we have $V = \pi_2(I_1)$.  Since $\Gr(k,\mathbb{F}^n)$ is projective,  $\pi_2$ is a closed map.  This leads to the following corollary of Lemma~\ref{lem-irreducible incidence}.
\begin{cor}\label{coro-irreducible and closed}
The set $V$ is an irreducible closed subvariety of $ \mathbb{P}(\Alt^{d}(\mathbb{F}^{n},\mathbb{F}^{m}))$.
\end{cor}

Given an integer $\max\{0,2k - n\} \le \ell \le k$,  we denote 
\begin{equation}\label{eq:Sigma}
\Sigma_\ell \coloneqq \big\lbrace
(\mathbb{U},\mathbb{V}) \in \Gr(k,\mathbb{F}^n) \times \Gr(k,\mathbb{F}^n): \dim \mathbb{U} \cap \mathbb{V} = \ell
\big\rbrace.
\end{equation}
\begin{lemma}\label{lem-intersection variety}
For each $\max\{0,2k - n\} \le \ell \le k$,  the set $\Sigma_\ell \subseteq \Gr(k,\mathbb{F}^n)\times \Gr(k,\mathbb{F}^n)$ is a quasi-projective subvariety of dimension $2k(n - k + \ell) - \ell(n+\ell)$.
\end{lemma}
\begin{proof}
By definition,  we have $\overline{\Sigma}_\ell  = \big\lbrace
(\mathbb{U},\mathbb{V}) \in \Gr(k,\mathbb{F}^n) \times \Gr(k,\mathbb{F}^n): \dim \mathbb{U} \cap \mathbb{V} \ge \ell\big\rbrace$. Thus,  $\Sigma_\ell = \overline{\Sigma}_\ell \setminus \overline{\Sigma}_{\ell+1}$ is a quasi-projective subvariety.  Next,  we define a map $\varphi: \Sigma_\ell \to \Gr(\ell,  \mathbb{F}^n)$ by $\varphi(\mathbb{U},  \mathbb{V}) = \mathbb{U} \cap \mathbb{V}$.  Since $2 k - n \le \ell$,  $\varphi$ is surjective.  Moreover,  given any $\mathbb{W} \in \Gr(\ell,  \mathbb{F}^n)$,  we have 
\[
\varphi^{-1}(\mathbb{W}) \simeq \big\lbrace
(\mathbb{U},  \mathbb{V}) \in \Gr(k-\ell,  \mathbb{F}^{n -\ell})\times \Gr(k-\ell,  \mathbb{F}^{n -\ell}):  \mathbb{U} \cap \mathbb{V} = \{0\}
\big\rbrace.
\]
In particular,
\[
\dim \varphi^{-1}(\mathbb{W})
= 2\dim \Gr(k-\ell,\mathbb{F}^{\,n-\ell})
= 2(k-\ell)\big((n-\ell)-(k-\ell)\big)
= 2(k-\ell)(n-k),
\]
where we used $\dim \Gr(r,\mathbb{F}^{N})=r(N-r)$ and that the condition
$\mathbb{U}\cap\mathbb{V}=\{0\}$ defines a nonempty Zariski open subset of
$\Gr(k-\ell,\mathbb{F}^{\,n-\ell})\times \Gr(k-\ell,\mathbb{F}^{\,n-\ell})$. Therefore,
\[
\dim \Sigma_\ell
= \dim \Gr(\ell,\mathbb{F}^n)+\dim \varphi^{-1}(\mathbb{W})
= \ell(n-\ell)+2(k-\ell)(n-k)
= 2k(n - k + \ell) - \ell(n+\ell).\qedhere
\]
\end{proof}
We will also need the following elementary inequality.  
\begin{lemma}\label{lem-inequality}
If $m,n,k,d,\ell \in \mathbb{N}$ satisfy
\[
n+\ell-2k\ge 0, \quad n>k\ge \ell\ge d\ge 2, \quad m\ge 2,\quad k(n-k)\ge m\binom{k}{d},
\]
then we have $\ell(n+\ell-2k)\ge m\binom{\ell}{d}$.
\end{lemma}
\begin{proof}
If $d \le 2 k - n $,  then $n-k \le k - d$,  $d \le k - 1$ and 
\[
m \le \frac{k (n-k)}{ \binom{k}{d}} = \frac{d (n-k)}{ \binom{k-1}{d-1}} \le \frac{d(k-d)}{ \binom{k-1}{d-1} } = \frac{d(d-1)(k-d)}{(k-d+1) \binom{k-1}{d-2}} < \frac{d(d-1)}{ \binom{d}{d-2}} = 2.
\]
This contradicts the assumption that $m \ge 2$. Therefore, we must have $d > 2k - n$. We observe that $\ell(n+\ell-2k)\ge m\binom{\ell}{d}$ is equivalent to $d(n+\ell-2k)\ge m\binom{\ell-1}{d-1}$.  We consider the function 
\[
h(x)=m\dfrac{(x-1)\cdots(x-d+1)}{(d-1)!}-d(n+x-2k).  
\]
We want to show that $h(\ell) \le 0$.  It is straightforward to verify that $h$ is a convex function on $[d,  +\infty)$.  Thus,  it holds that 
\[
h(\ell) \le \max \{ h(d),  h(k) \} = \max \Big\lbrace
m - d(n + d - 2k),  m \binom{k-1}{d-1} - d (n - k)
\Big\rbrace.
\]
Since $k (n-k) \ge m \binom{k}{d}$,  we may derive $m \binom{k-1}{d-1} - d (n - k) \le 0$.  This also implies
\[
m - d(n + d - 2k) \le \frac{d(n - k)}{\binom{k-1}{d-1}} - d(n+ d - 2k) = \frac{d}{\binom{k-1}{d-1}} 
\Big[ 
(n-k) - \binom{k-1}{d-1}(n + d -2k)
\Big].
\]
Now that $d > 2k - n$,  we have $n + d \ge 2k + 1$,  $n - k \ge 1$.  Let $u = 1$,  $v = n - k > 0$,  $s = n + d - 2k > 0$ and $t = k - d + 1 > 0$.  Then $u + v = s + t$ and $u \le s,  t \le v$,  from which we obtain 
\[
n - k = u v \le st = (k - d + 1) (n+d - 2k) \le \binom{k-1}{d-1} (n+d - 2k).  
\]
Consequently, we establish that $h(\ell) \le 0$, as desired. 
\end{proof}

The following theorem corresponds to Theorem~\ref{main theorem-1}\ref{main theorem-1:item1}.
\begin{theorem}\label{main theorem-1a}
If $\mathbb{F}$ is algebraically closed and $m \ge 2$, then $\alpha_{\w}(\mathbb{F},n,d,m) = \max \big\lbrace 
s \in \mathbb{N}_0: s(n-s)\ge m\binom{s}{d}
\big\rbrace$.
\end{theorem}

\begin{proof} 
It is sufficient to compute $\alpha_\w(\mathbb{F},n,d,m)$ for $2 \le d \le n$, since $\alpha_\w(\mathbb{F},n,1,m) = \max \{0,n-m\}$ and $\alpha_\w(\mathbb{F},n,d,m) = n$ if $d > n$. 

Suppose $m \ge 2$. Denote $k_0 \coloneq \max \big\lbrace s \in \mathbb{N}_0: s(n-s) \ge m \binom{s}{d}\big\rbrace$.  According to the discussion in Section~5 of \cite{FP92},  $k \coloneqq \alpha_{\w}(\mathbb{F},n,d,m)$ must satisfy the inequality $k (n-k) \ge m \binom{k}{d}$.  This implies $\alpha_{\w}(\mathbb{F},n,d,m) \le k_0$.  It is thus left to prove the reversed inequality.  Let $I_1 \subseteq \Gr(k_0,\mathbb{F}^n) \times \mathbb{P} ( \Alt^d(\mathbb{F}^n,  \mathbb{F}^m) )$ and $V \subseteq \mathbb{P} ( \Alt^d(\mathbb{F}^n,  \mathbb{F}^m) )$ be the sets respectively defined by \eqref{eq:incidence} and \eqref{eq:V} for parameters $n,d,m,  k_0$.  According to Lemma~\ref{lem-irreducible incidence} and Corollary~\ref{coro-irreducible and closed},  $I_1$ and $V$ are irreducible projective varieties and $V = \pi_2 (I_1)$ where $\pi_2 :I_1 \to \mathbb{P} ( \Alt^d(\mathbb{F}^n,  \mathbb{F}^m) )$ is the projection map to the second component. 

By  Lemma~\ref{dimension of fibre},  we have $\dim \pi_2^{-1} ([T]) = \dim I_1 - \dim V$ for a generic $[T] \in V$.  The inequality $\alpha_{\w}(\mathbb{F},n,d,m) \ge k_0$ is equivalent to $V = \mathbb{P} ( \Alt^d(\mathbb{F}^n,  \mathbb{F}^m) )$.  We proceed by contradiction.  Assume that $V \ne \mathbb{P} ( \Alt^d(\mathbb{F}^n,  \mathbb{F}^m) )$.  We consider 
\[
I_2 \coloneqq \big\lbrace
(\mathbb{U},\mathbb{V},  [T]) \in \Gr(k_0,\mathbb{F}^n) \times \Gr(k_0,\mathbb{F}^n) \times \mathbb{P} ( \Alt^d(\mathbb{F}^n,  \mathbb{F}^m) ): T|_{\mathbb{U}} = 0,\; T|_{\mathbb{V}} = 0
\big\rbrace.
\]
Let $\psi_3: I_2 \to \mathbb{P} ( \Alt^d(\mathbb{F}^n,  \mathbb{F}^m) )$ be the projection map to the third component.  Then clearly we have $\psi_3(I_2) = V$ and for any $[T] \in V$,  it holds that $\dim \psi_3^{-1}([T])  = 2 \dim \pi_2^{-1}([T]) $. Thus,  Lemma~\ref{dimension of fibre} implies
\begin{equation}\label{main theorem-1:eq1}
\dim I_2 \ge \dim \psi_3^{-1}([T]) + \dim V 
\ge 2\dim \pi_2^{-1}([T]) + \dim V = 2 \dim I_1  - \dim V
\end{equation}
for a generic $[T] \in V$.  Lemma~\ref{lem-irreducible incidence} and the assumption that $V \ne \mathbb{P} (\Alt^d(\mathbb{F}^n,\mathbb{F}^m))$ imply 
\[
\dim I_1 = (n - k_0)k_0 + m \bigg[ \binom{n}{d} - \binom{k_0}{d} \bigg] - 1,\quad \dim V \le m\binom{n}{d} - 2.
\]
Therefore,  we may obtain from \eqref{main theorem-1:eq1} that 
\begin{equation}\label{main theorem-1:eq2}
\dim I_2 \ge 2 (n - k_0)k_0 + m \binom{n}{d} - 2m \binom{k_0}{d}
\end{equation}

On the other hand,  let $\psi_{12}: I_2 \to \Gr(k_0,\mathbb{F}^n) \times \Gr(k_0,\mathbb{F}^n)$ be the projection map to the first two components.  We partition $\Gr(k_0,\mathbb{F}^n) \times \Gr(k_0,\mathbb{F}^n)$ as $\Gr(k_0,\mathbb{F}^n) \times \Gr(k_0,\mathbb{F}^n) = \bigsqcup_{\ell = \max\{0, 2k_0-n\}}^{k_0} \Sigma_\ell$,  where for each $\max\{0, 2k_0-n\} \le \ell \le k_0$,  $\Sigma_\ell$ is defined as in \eqref{eq:Sigma}.  Note that for each $(\mathbb{U},\mathbb{V})\in \Sigma_\ell$, we have 
\[
\dim \psi_{12}^{-1} (\mathbb{U},\mathbb{V}) = m \bigg[ \binom{n}{d} - 2 \binom{k_0}{d} + \binom{\ell}{d} \bigg] - 1.
\]
Thus, Lemma~\ref{dimension of fibre} together with Lemma~\ref{lem-intersection variety} leads to 
\begin{align}
\dim I_2 &= \max_{\max \{0, 2k_0-n\} \le \ell  \le k_0} \big\lbrace \dim \psi_{12}^{-1}(\mathbb{U},  \mathbb{V}) + \dim \Sigma_\ell  \big\rbrace \nonumber \\
&= \max_{\max \{0, 2k_0-n\} \le \ell  \le k_0}  \bigg\lbrace
m\bigg[ \binom{n}{d}-2\binom{k_{0}}{d}+\binom{\ell}{d}\bigg] - 1 + 2k_0(n-k_0+\ell)-\ell(n+\ell)
 \bigg\rbrace.  \nonumber \\
 &=2k_0(n-k_0)  + m\bigg[\binom{n}{d}-2\binom{k_0}{d}\bigg] - 1 +  \max_{\max \{0, 2k_0-n\} \le \ell  \le k_0}  \bigg\lbrace 
m\binom{\ell}{d}-\ell(n+\ell-2k_0)
 \bigg\rbrace.  \label{main theorem-1:eq3}
\end{align}
If $\ell < d$,  then $\binom{\ell}{d} = 0$ and $m\binom{\ell}{d}-\ell(n+\ell-2k_0) \le 0$.  If $\ell \ge d$,  then since $k_0(n - k_0) \ge m \binom{k_0}{d}$,  $n + \ell - 2k_0 \ge 0$,  $n > k_0 \ge \ell$ and $d,  m \ge 2$, Lemma~\ref{lem-inequality} shows $m\binom{\ell}{d}-\ell(n+\ell-2k_0) \le 0$.  Therefore,  we arrive at a contradiction by virtue of \eqref{main theorem-1:eq2} and \eqref{main theorem-1:eq3}.
\end{proof}

Next, we prove Theorem~\ref{main theorem-1}\ref{main theorem-1:item2}. To this end, we establish some lemmas.

\begin{lemma}\label{n-2 case}
Suppose $n \ge 3$. There exists $T \in \Alt^{n-2}(\mathbb{F}^n,\mathbb{F})$ such that $\alpha_{\Lambda}(T)\le n-2$ if and only if $n$ is even.
\end{lemma}
\begin{proof}
Denote $\mathbb{V} \coloneqq (\mathbb{F}^n)^\ast$. Then there is a canonical isomorphism $\varepsilon: \Alt^{n-2}(\mathbb{F}^n,\mathbb{F}) \to \Lambda^{n-2} \mathbb{V}$ given by 
\[
T(u_1,\dots, u_{n-2}) = \langle \varepsilon(T), u_1 \wedge \cdots \wedge u_{n-2} \rangle,\quad u_1,\dots, u_{n-2} \in \mathbb{F}^n.
\]
By choosing a basis of $\mathbb{V}$, we may identify $\Lambda^n \mathbb{V}$ with $\mathbb{F}$ to obtain a linear map $\tau: \Alt^{n-2}(\mathbb{F}^n,\mathbb{F}) \to \Alt^2(\mathbb{V},\mathbb{F})$ where $\tau(T) \in \Alt^2(\mathbb{V},\mathbb{F})$ is defined by 
\[
\tau(T)(\eta_1,\eta_2) \coloneqq \varepsilon(T) \wedge \eta_1 \wedge \eta_2 \in \Lambda^n (\mathbb{V}) \cong \mathbb{F},\qquad \eta_1,\eta_2 \in \mathbb{V}.
\]  
Clearly, both $\varepsilon$ and $\tau$ are isomorphisms of vector spaces.

Given a nonzero $T \in \Alt^{n-2}(\mathbb{F}^n,\mathbb{F})$, we have $\alpha_{\Lambda}(T)= n-1$ if and only if there is some $\mathbb{U} \in \Gr(n-1,\mathbb{F}^n)$ such that $\langle \varepsilon(T), \Lambda^{n-2} \mathbb{U} \rangle = 0$. The latter is equivalent to $\varepsilon(T) = \eta \wedge T_1$ for some $\eta \in \mathbb{U}^{\perp} \subseteq \mathbb{V}$ and $T_1 \in \Lambda^{n-3} \mathbb{V}$. This implies that $\alpha_{\Lambda}(T)=n-1$ if and only if $\eta \wedge \varepsilon(T) = 0$ for some nonzero $\eta \in \mathbb{V}$. Consequently, $T \in \Alt^{n-2}(\mathbb{F}^n,\mathbb{F})$ satisfies $\alpha_{\w}(T) \le n -2$ if and only if $\tau(T)$ is a non-degenerate alternating bilinear form on $\mathbb{V}$, which exists if and only if $n$ is even.
\end{proof}

\begin{lemma}\label{7-3 case}
For any algebraically closed field $\mathbb{F}$, we have $\alpha_{\Lambda}(\mathbb{F}, 7,3,1)\le 4$.
\end{lemma}
\begin{proof}
We split our discussion into two cases.

\textbf{Case 1: $\mathrm{char}(\mathbb{F})\neq 2$.}
Let $\mathbb{O}$ be the octonion algebra over $\mathbb{F}$. We write $\mathbb{O} = \mathbb{F} \oplus \mathbb{V}$, where $\mathbb{V} \cong \mathbb{F}^7$ is the subspace of pure octonions in $\mathbb{O}$. Define a trilinear map
\[
T:\mathbb{V}\times \mathbb{V}\times \mathbb{V}\to \mathbb{F},\quad
T(v_1,v_2,v_3) = \langle v_1\times v_2, v_3 \rangle,
\]
where $u\times v := uv-vu$, and $\langle u,v\rangle \coloneqq u\bar{v}+ v \bar{u}$. We show that $\alpha_\w(T) \le 4$. Assume for contradiction that $T$ has a $5$-dimensional isotropic subspace $\mathbb{U}\subseteq \mathbb{V}$. Then there is a bilinear map
\[
f_T:\mathbb{U}\times \mathbb{U} \to \mathbb{U}^\perp,\qquad f_T(v_1,v_2)=v_1\times v_2,
\]
where $\mathbb{U}^\perp$ is the orthogonal complement of $\mathbb{U}$ with respect to $\langle\cdot,\cdot\rangle$. If $\langle u, u \rangle = 0$ for any $u \in \mathbb{U}$, then we may deduce that $\mathbb{U} \subseteq \mathbb{U}^{\perp}$. This yields a contradiction since $\dim \mathbb{U}^{\perp} = 2$ by the non-degeneracy of $\langle\cdot,\cdot\rangle$. Thus, there exists $u_0\in \mathbb{U}\setminus \mathbb{U}^{\perp}$. By definition, we have $\langle u_0,u_0\rangle \neq 0$. Set $\mathbb{W} \coloneqq \mathbb{U} \cap u_0^\perp$. As $u_0^\perp$ is a hyperplane in $\mathbb{V}$, we have $\dim \mathbb{W} =4$. We consider the linear map
\[
\iota: \mathbb{W} \to \mathbb{U}^\perp,\quad \iota(w) = u_0 \times w.
\]
The octonion identity implies:
\[
u_0 \times (u_0 \times w)
= 2\big(\langle u_0,w\rangle u_0 - \langle u_0,u_0\rangle w\big)
= -2\langle u_0,u_0\rangle w.
\]
Since $\langle u_0,u_0\rangle\neq 0$ and $\mathrm{char}(\mathbb{F})\neq 2$, the map $\iota$ is injective. However, we have $\dim \mathbb{W}=4$ and $\dim \mathbb{U}^\perp = 2$, implying that $\iota$ is not injective. 

\textbf{Case 2: $\mathrm{char}(\mathbb{F})=2$.} We prove $\alpha_{\Lambda}(\mathbb{F},7,3,1) \le 4$ by a direct calculation. Let $\{\eta_1,\dots,\eta_7\}$ be the standard basis of $(\mathbb{F}^7)^\ast$. We consider the tensor $T\in \Alt^3 (\mathbb{F}^7, \mathbb{F})$ defined as:
\[
T
= \eta_1\wedge \eta_2\wedge \eta_5
+ \eta_1\wedge \eta_3\wedge \eta_6
+ \eta_1\wedge \eta_4\wedge \eta_7
+ \eta_2\wedge \eta_3\wedge \eta_4
+ \eta_5\wedge \eta_6\wedge \eta_7.
\]
The Grassmannian $\Gr(5,\mathbb{F}^7)$ has the decomposition $\Gr(5,\mathbb{F}^7) = \bigcup_{\substack{I\subseteq [7]\\ |I|=5}} U_I$, where $U_I$ is the Schubert cell indexed by $I$. The alternating form $T$ has a $5$-dimensional isotropic subspace if and only if there exists some $I$ and $\mathbb{V}\in U_I$ such that $T|_\mathbb{V}=0$. This is equivalent to the solvability of a polynomial system on $U_I\cong \mathbb{F}^{r_I}$. Here $0 \le r_I \le 10$ is the dimension of $U_I$. We use Gr\"obner basis over $\mathbb{F}_2$ to test the solvability of these systems.

The Python code in Appendix~\ref{sec:appendB} verifies that for every $I\subseteq [7]$ with $|I| = 5$, the ideal generated by the corresponding polynomial system is the entire polynomial ring. In fact, $20$ of the $21$ polynomial systems yield the contradictory equation $1=0$, while the Gröbner basis of the remaining system is $\{1\}$. Therefore, $T$ has no $5$-dimensional isotropic subspace.
\end{proof}

The theorem that follows corresponds to \ref{main theorem-1:item2} of Theorem~\ref{main theorem-1}.
\begin{theorem}\label{main theorem-1b}
If $\mathbb{F}$ is algebraically closed, then 
\[
\alpha_{\w}(\mathbb{F},n,d,1) = \begin{cases}
\lceil \frac{n}{2} \rceil \quad &\text{if~} d = 2, \\
4 \quad &\text{if~} (d,n) = (3,7),  \\
n - 2 \quad &\text{if~} d = n-2 \text{~is even},  \\
\max \big\lbrace 
s \in \mathbb{N}_0: s(n-s)\ge \binom{s}{d}
\big\rbrace \quad &\text{otherwise}. 
\end{cases}
\]
\end{theorem}
\begin{proof}
According to \cite[Section~5]{FP92}, Lemmas~\ref{n-2 case} and \ref{7-3 case}, we obtain \[
\alpha_{\w}(\mathbb{F},n,d,1) \le 
\begin{cases}
\lceil \frac{n}{2} \rceil \quad &\text{if~} d = 2, \\
4 \quad &\text{if~} (d,n) = (3,7),  \\
n - 2 \quad &\text{if~} d = n-2 \text{~is even},  \\
\max \big\lbrace 
s \in \mathbb{N}_0: s(n-s)\ge \binom{s}{d}
\big\rbrace \quad &\text{otherwise}. 
\end{cases}
\]
By \cite[Theorem~1]{Tevelev01}, equality holds if $\ch (\mathbb{F}) = 0$.

To establish the equality for $\mathbb{F}$ with $\ch (\mathbb{F}) > 0$, we denote by $W(\mathbb{F})$ the ring of Witt vectors of $\mathbb{F}$, and write $\mathbb{K}$ for the field of fractions of $W(\mathbb{F})$. Since $\ch (\mathbb{K}) = 0$, it is sufficient to prove that $k \coloneqq \alpha_{\Lambda}(\overline{\mathbb{K}},n,d,1) \le \alpha_{\Lambda}(\mathbb{F},n,d,1)$. For each $T \in \Alt^d(\mathbb{F}^n,\mathbb{F})$, we show that $T$ has a $k$-dimensional isotropic subspace. Let $\pi: W(\mathbb{F}) \to W(\mathbb{F})/\mathfrak{m} \cong \mathbb{F}$ be the natural projection map, where $\mathfrak{m}$ is the maximal ideal of $W(\mathbb{F})$. By definition, there exists some $\widetilde{T}\in \Alt^d(W(\mathbb{F})^n, W(\mathbb{F}))$ such that $\widetilde{T} \equiv T \pmod{\mathfrak{m}}$, and any rank $k$ direct summand of $W(\mathbb{F})^n$ that are isotropic to $\widetilde{T}$ induces a $k$-dimensional isotropic subspace of $T$. Let $\mathcal{X}_{\widetilde{T}}$ be the closed subscheme of the Grassmannian over $W(\mathbb{F})$ consisting of rank $k$ direct summands of $W(\mathbb{F})^n$ that is isotropic to $\widetilde{T}$. In particular, $\mathcal{X}_{\widetilde{T}}$ is proper over $\operatorname{Spec} W(\mathbb{F})$.

By assumption, $\mathcal{X}_{\widetilde{T}}(\overline{\mathbb{K}}) \ne \emptyset$. Thus, there exists some finite extension $\mathbb{L}/\mathbb{K}$ such that $\mathcal{X}_{\widetilde{T}}(\mathbb{L}) \ne \emptyset$. By Lemma~\ref{lem:properness}, we obtain $\mathcal{X}_{\widetilde{T}}(\mathfrak{O}_{\mathbb{L}}) \ne \emptyset$, where $\mathfrak{O}_{\mathbb{L}}$ is the ring of integers of $\mathbb{L}$. Since $\mathbb{K}$ is complete non-archemidean and $\mathbb{F}$ is algebraically closed, Lemma~\ref{property of local field} implies that $\kappa(\mathbb{L}) = \kappa (\mathbb{K}) = W(\mathbb{F})/\mathfrak{m}\cong \mathbb{F}$. Therefore, a point in $\mathcal{X}_{\widetilde{T}}(\mathfrak{O}_{\mathbb{L}})$ induces a $k$-dimensional isotropic subspace of $T$ in $\mathbb{F}^n$ and this completes the proof.
\end{proof}

\subsection{Tur\'{a}n Problem for $\Hom^d(\mathbb{F}^{n}, \mathbb{F}^m)$}\label{subsec:MultilinearTuran}
We prove Theorem~\ref{main theorem-2} in this subsection.  To begin with,  we consider 
\begin{align}
J_1 &\coloneqq \big\lbrace (\mathbb{U}_1,\dots, \mathbb{U}_{d}, [T])\in \Gr(k, \mathbb{F}^n)^d \times \mathbb{P}( \Hom^d(\mathbb{F}^{n},  \mathbb{F}^m)):  T|_{\mathbb{U}_1 \times \cdots \times \mathbb{U}_d}=0 \big\rbrace,  \label{eq:J1} \\ 
W &\coloneqq \big\lbrace
[T] \in \mathbb{P}(\Hom^d(\mathbb{F}^{n},  \mathbb{F}^m)): \alpha(T) \ge k
\big\rbrace.  \label{eq:W}
\end{align}
The following lemma is an analogue of Lemma~\ref{lem-irreducible incidence}, with an identical proof that is thus omitted.
\begin{lemma}\label{lem:dim J1}
Suppose $k,n$ are integers such that $0 \le k \le n$. The set $J_1$ is an irreducible closed subset of $\Gr(k, \mathbb{F}^n)^d \times \mathbb{P}(\Hom^d(\mathbb{F}^{n},  \mathbb{F}^m))$,  whose dimension is $kd(n-k)+m\left(n^{d}-k^{d}\right)-1$. Let $\pi_2: \Gr(k, \mathbb{F}^n)^d \times \mathbb{P}(\Hom^d(\mathbb{F}^{n},  \mathbb{F}^m)) \to \mathbb{P}(\Hom^d(\mathbb{F}^{n},  \mathbb{F}^m))$ be the natural projection map.  Then $W = \pi_2 (J_1)$,  and it is an irreducible closed subvariety of $\mathbb{P}(\Hom^d(\mathbb{F}^{n},  \mathbb{F}^m))$. 
\end{lemma}
The idea underlying the proof of Theorem~\ref{main theorem-2} is the same as that for Theorem~\ref{main theorem-1}, but the detailed calculations are different,  which we include below for completeness.

\begin{proof}[Proof of Theorem~\ref{main theorem-2}]
The rank-nullity theorem implies $\alpha(\mathbb{F},n,1,m) =\max\{ 0, n - m\}$. Thus, we may assume $d\ge 2$. For an integer $0 \le k \le n$, let $J_1$ and $W$ be the varieties defined in \eqref{eq:J1} and \eqref{eq:W}, respectively. 

We first prove \ref{main theorem-2:item1} for $d \ge 2$, together with \ref{main theorem-2:item2} for $d \ge 3$. We take $k = \alpha(\mathbb{F},n,d,m)$ and prove $dk(n-k) \ge m k^d$. By definition, we have $W = \mathbb{P}(\Hom^d(\mathbb{F}^{n}, \mathbb{F}^m))$. Let $\pi_1: J_1 \to \Gr(k,\mathbb{F}^n)^d$ be the natural projeciton onto the first factor. Then the fiber of $\pi_1$ is a linear subspace of $\mathbb{P}(\Hom^d(\mathbb{F}^n, \mathbb{F}^m))$ of codimension $m k^d$. This implies that $J_1$ is an irreducible variety of dimension $dk(n-k)+m(n^d-k^d)-1$. The desired inequality follows immediately from $\dim J_1 \ge \dim W = mn^d - 1$.

Conversely, suppose either $m = 1$ and $d\ge 3$ or $m \ge 2$. Let $k$ be the maximal integer satisfying $dk(n-k) \ge m k^d$. We want to show that $\alpha(\mathbb{F},n,d,m) \ge k$, or equivalently, that $W = \mathbb{P}(\Hom^d(\mathbb{F}^{n}, \mathbb{F}^m))$. We observe that $n \ge 2k$. This is clear if $k \le 1$. Assume $k \ge 2$ and $n < 2k$. This implies 
\begin{equation}\label{main theorem-2:eq}
    m k^{d-1}\le d(n-k)\le d(k-1).
\end{equation}
Since $d(k-1) < 2k^{d-1}$, \eqref{main theorem-2:eq} implies that $m < 2$. Thus, \eqref{main theorem-2:eq} holds only if $m = 1$ and $d \ge 3$. However, in this case, we have $k^{d-1}/(k-1) \ge 4 k^{d-3} \ge 2^{d-1} > d$, which contradicts \eqref{main theorem-2:eq}.

If $W \ne \mathbb{P}(\Hom^d(\mathbb{F}^{n}, \mathbb{F}^m))$, then $\dim W \le mn^d - 2$. We consider a subset $J_2$ of $X \coloneqq \Gr(k,\mathbb{F}^n)^d\times \Gr(k,\mathbb{F}^n)^d \times \mathbb{P}(\Hom^d(\mathbb{F}^{n}, \mathbb{F}^m))$ defined as 
\[
J_2 = \{(\mathbb{U}_1,\dots,\mathbb{U}_d,\mathbb{V}_1,\dots, \mathbb{V}_d,[T]) \in X:
T|_{\mathbb{U}_1\times\cdots\times\mathbb{U}_d}=0,\; T|_{\mathbb{V}_1\times\cdots\times\mathbb{V}_d}=0\}.
\]
Let $\rho_{12}: J_2 \to \Gr(k,\mathbb{F}^n)^d\times \Gr(k,\mathbb{F}^n)^d$ and $\rho_3: J_2 \to \mathbb{P}(\Hom^d(\mathbb{F}^{n}, \mathbb{F}^m))$ be natural projection maps. We note that $\rho_3(J_2) = \pi_2(J_1) = W$. Moreover, the same argument as in the proof of Theorem~\ref{main theorem-1} implies that
\begin{equation}\label{eq:LJk2}
\dim J_2 \ge 2 \dim J_1 - \dim W \ge 2dk(n-k)+mn^d-2mk^d.
\end{equation}
On the other side, for each $0 \le \ell \le k$, we define $\Sigma_\ell=\{(\mathbb{U},\mathbb{V})\in\Gr(k,\mathbb{F}^n)^2:
\dim(\mathbb{U}\cap\mathbb{V})=\ell\}$. Since $n \ge 2k$, Lemma~\ref{lem-intersection variety} implies $\dim\Sigma_\ell=2k(n-k+\ell)-\ell(n+\ell)$. For any $0 \le \ell_1,\dots, \ell_d \le k$, the fiber $\rho_{12}^{-1}(\mathbb{U}_1,\dots, \mathbb{U}_d, \mathbb{V}_1,\dots, \mathbb{V}_d)$ is a linear subspace of $\mathbb{P}(\Hom^d(\mathbb{F}^{n}, \mathbb{F}^m))$ of codimension $m(2k^d - \prod_{i=1}^d \ell_i)$. Therefore, we obtain
\begin{equation}\label{eq:secondJK}
    \dim J_2
\le
2dk(n-k)+m(n^d-2k^d)-1
+\max_{0\le \ell_1,\dots,\ell_d \le k}
\Big\{
m\prod_{i=1}^d\ell_i-\sum_{i=1}^d\ell_i(n+\ell_i-2k)
\Big\}.
\end{equation}
We claim that $m\prod_{i=1}^d\ell_i-\sum_{i=1}^d\ell_i(n+\ell_i-2k) \le 0$. This together with \eqref{eq:LJk2}, contradicts \eqref{eq:secondJK}. 

The claim is clearly true for $k = 0$. To prove the claim for $k \ge 1$, we take $a = (n-2k)/k$ and $x_i = \ell_i/k$, $1\le i \le d$. Then
\begin{equation}\label{eq:axi-1}
m\prod_{i=1}^d \ell_i = mk^d \prod_{i=1}^d x_i \le d k(n - k) \prod_{i=1}^d x_i = k^2 d(a+1) \prod_{i=1}^d x_i.
\end{equation}
Since $0 \le a$ and $0 \le x_i \le 1$ for each $1 \le i \le d$, we have $(a+1)x_i \le a + x_i$. By $d \ge 2$ and the AM-GM inequality, we deduce 
\begin{equation}\label{eq:axi-2}
d(a+1) \prod_{i = 1}^d x_i\le d(a+1) \Big(\prod_{i = 1}^d x_i \Big)^{2/d} \le d \Big(\prod_{i = 1}^d (a+x_i)x_i \Big)^{1/d} \le \sum_{i=1}^d (a + x_i)x_i.
\end{equation} 
Combining \eqref{eq:axi-1} and \eqref{eq:axi-2} completes the proof of the claim.

Next, we prove \ref{main theorem-2:item2} with $d \le 2$. Since $\Hom^1(\mathbb{F}^{n}, \mathbb{F})$ consists of linear functions on $\mathbb{F}^{n}$, we clearly have $\alpha(\mathbb{F},n,1,1) = n- 1$. It is left to compute $\alpha(\mathbb{F},n,2,1)$. Given $T \in \Hom^2(\mathbb{F}^{n}, \mathbb{F})$ and $\mathbb{V} \subseteq \mathbb{F}^n$, we have $\dim \mathbb{V}^{\perp} \ge n - \dim \mathbb{V}$ where $\mathbb{V}^{\perp} \coloneqq \{u\in \mathbb{F}^n: T(u,v) = 0,\; v\in \mathbb{V}\}$. Denote $k \coloneqq \alpha(\mathbb{F},n,2,1)$. If $T$ is a non-degerate bilinear form and $T|_{\mathbb{V}_1 \times \mathbb{V}_2} = 0$ for some $k$-dimensional subspaces $\mathbb{V}_1$ and $\mathbb{V}_2$, then $\mathbb{V}_2 \subseteq \mathbb{V}_1^\perp$ and $k \le n - k$. This implies $k \le \lfloor n/2 \rfloor$. Conversely, let $\mathbb{U} \subseteq \mathbb{F}^n$ be a fixed subspace of dimension $\lfloor n/2 \rfloor$. For each $T \in \Hom^2(\mathbb{F}^{n}, \mathbb{F})$, we have $\dim \mathbb{U}^\perp \ge \lceil n/2 \rceil$. Thus, any $\lfloor n/2 \rfloor$-dimensional subspace $\mathbb{V} \subseteq \mathbb{U}$ satisfies $T|_{\mathbb{V} \times \mathbb{U}} = 0$. This implies $k \ge \lfloor n/2 \rfloor$, and completes the proof.
\end{proof}

\section*{Acknowledgements}
The authors would like to thank Dilong Yang and Yijia Fang for helpful discussions, especially Dilong Yang for many comments on an early draft. As a participant in the ECOPRO Summer Student Research Program, Qiyuan Chen acknowledges the organizers, especially Hong Liu, and the members of IBS for their support and hospitality. Zixiang Xu would like to thank Qiyuan Chen and Ke Ye for their hospitality at the Chinese Academy of Sciences. During the early stage of this work, Zixiang Xu was supported by IBS-R029-C4.

\bibliographystyle{abbrv}
\bibliography{Construction}

\begin{appendices}
\section{Lower Bound for the Erd\H{o}s Box Problem and the Barrier}\label{sec:append}
To quantify the complexity of multilinear tensors, Gowers and Wolf~\cite{2011Gowers} introduced the notion of \emph{analytic rank} in the context of higher-order Fourier analysis. This notion was further developed and studied systematically by Lovett~\cite{lovett2018analytic}, and can be viewed as a nonlinear analogue of matrix rank for higher-order tensors.  On the other hand, the \emph{partition rank} is an algebraic complexity measure for tensors, originating from additive combinatorics and playing a central role in a variety of problems in extremal and additive combinatorics~\cite{2017AnnZ4,2017AnnF3,2020JAC,2020JCTAPrank,2023MathkaPrank} and is closely related to analytic rank see~\cite{2020Janzer,lovett2018analytic,2019GAFA}.

Now Let $d,n,m$ be positive integers and let $q$ be a prime power. Let $T \in \Hom^d(\mathbb{F}_q^{n+1},\mathbb{F}_q^m)$.  The \emph{analytic rank}~\cite{2011Gowers} can be defined as
\[
\operatorname{AR}_{\mathbb{F}_q}(T)\coloneqq d(n+1)-\log_q |Z_T|,
\]
where
\(
Z_T \coloneqq \{(x_1,\dots,x_d)\in(\mathbb{F}_q^{n+1})^d:\ T(x_1,\dots,x_d)=0\}.
\)
According to~\cite[Theorem~1.7]{lovett2018analytic}, we have
\[
\operatorname{AR}_{\mathbb{F}_q}(T)\le \operatorname{PR}_{\mathbb{F}_q}(T),
\]
where $\operatorname{PR}_{\mathbb{F}_q}(T)$ denotes the partition rank of $T$.  Since $T$ takes values in $\mathbb{F}_q^m$, we trivially have $\operatorname{PR}_{\mathbb{F}_q}(T)\le m$ by decomposing $T$ into its $m$ coordinate tensors, and hence the analytic rank admits the following uniform upper bound.
  
\begin{lemma}\label{AR-upper-bound}
For any $T \in \Hom^d(\mathbb{F}_q^{n+1},  \mathbb{F}_q^m)$,  $\operatorname{AR}_{\mathbb{F}_q}(T)  \le  m$.  
\end{lemma}

Let $\mathcal{G}(T) = (V_T,  E_T)$ be the hypergraph defined by \eqref{eq:Erdox} for $T \in \Hom^d(\mathbb{F}_q^{n+1},  \mathbb{F}_q^m)$.
\begin{lemma}\label{lem:VT}
For fixed positive integers $d,n$ and $m$, the following holds:
\[
|V_T| = \frac{d (q^{n+1}-1)}{q-1},\quad  |E_T| \ge \frac{q^{d(n+1)-m}-dq^{(d-1)(n+1)}}{(q-1)^d}=(1-o(1))q^{dn-m}.
\]
\end{lemma}
\begin{proof}
Since $V_T$ is the disjoint union of $d$ copies of $\mathbb{P}^n(\mathbb{F}_q)$,  
\[
|V_T|  = d |\mathbb{P}^n(\mathbb{F}_q)| = \frac{d (q^{n+1}-1)}{q-1}.
\]
By definition and Lemma~\ref{AR-upper-bound},  we have $|Z_T| \geq q^{d(n+1)-\mathrm{AR}_{\mathbb{F}_q}(T)}\geq q^{d(n+1)-m}$.  Since 
\[
E_T = \big( Z_T \setminus \{ (x_1,\dots,  x_d)\in (\mathbb{F}_q^{n+1})^d:  x_i = 0\text{~for some~} 1 \le i \le d \} \big) / (\mathbb{F}_q^{\times})^d \subseteq \mathbb{P}^n(\mathbb{F}_q)^d,  
\]
we obtain 
\[
|E_T| \geq \frac{|Z_T|-dq^{(d-1)(n+1)}}{(q-1)^d}\geq \frac{q^{d(n+1)-m}-dq^{(d-1)(n+1)}}{(q-1)^d}=(1-o(1))q^{dn-m}.  \qedhere
\]
\end{proof}
Let $D_T$ be the subset of $\Gr(2,\mathbb{F}_q^{n+1})^d$ consisting of all $d$-tuples $(\mathbb{V}_1,\dots , \mathbb{V}_d)$ such that $T|_{\mathbb{V}_1 \times \cdots \times \mathbb{V}_d} = 0$.  
\begin{lemma}\label{lem:count_bad_configuration}
Let $n,d\ge 2$ be positive integers.  For any sufficiently large prime power $q$,  there exists $T\in \Hom^{d}(\mathbb{F}_{q}^{n+1},\mathbb{F}_{q}^{m})$ such that 
\begin{equation}\label{eq:DT}
|D_T|\leq q^{-m2^{d}}\dfrac{(q^{n+1}-1)^{d}(q^{n+1}-q)^{d}}{(q^{2}-1)^{d}(q^{2}-q)^{d}}=(1+o(1))q^{2(n-1)d-m2^{d}}.
\end{equation}
\end{lemma}
\begin{proof}
Let $J_1$ be the subvariety of $\Gr(2,\mathbb{F}_q^{n+1})^{d} \times \mathbb{P} ( \Hom^{d}(\mathbb{F}_{q}^{n+1},\mathbb{F}^{m}_{q}) )$ defined in \eqref{eq:J1} and let $\pi_1: J_1 \to \Gr(2,\mathbb{F}_{q}^{n+1})^d$ be the natural projection map.  Given any $(\mathbb{V}_1,\dots , \mathbb{V}_d) \in \Gr(2,\mathbb{F}_{q}^{n+1})^{d}$,  we have $|\pi_1^{-1}(\mathbb{V}_1,\dots , \mathbb{V}_d)| = q^{m( (n+1)^{d}-2^{d})}/(q-1)$.  Thus by \cite[Proposition~1.7.2]{stanley2011enumerative},  it holds that
\[
\vert J_1 \vert= \frac{q^{m( (n+1)^{d}-2^{d})}}{q-1} \vert \Gr(2,\mathbb{F}_{q}^{n+1})^{d}\vert= \frac{q^{m( (n+1)^{d}-2^{d})}}{q-1}\dfrac{(q^{n + 1}-1)^{d}(q^{n + 1}-q)^{d}}{(q^{2}-1)^{d}(q^{2}-q)^{d}}.  
\]
By the Pigeonhole principle,  there must exist $[T]\in\Hom^{d}(\mathbb{F}_{q}^{n+1},\mathbb{F}^{m}_{q})$ such that 
\[
|D_T| \le \frac{|J_1|}{|\mathbb{P} (\Hom^d(\mathbb{F}_q^{n+1},  \mathbb{F}_q^m))|}  = q^{-m2^{d}}\dfrac{(q^{n+1} -1)^{d}(q^{n+1}-q)^{d}}{(q^{2}-1)^{d}(q^{2}-q)^{d}}.  \qedhere
\]
\end{proof}
Let $T$ be a multilinear map satisfying \eqref{eq:DT}. According to the multilinear method described in Subsection~\ref{subsubsec:Erdos},  we construct the hypergraph $\mathcal{G}'(T)$ by deleting the hyperedges of $\mathcal{G}(T)$ corresponding to points in \(\bigcup_{(\mathbb{V}_1,\dots,\mathbb{V}_d)\in D_T}
\Big(\mathbb{P}(\mathbb{V}_1)\times\cdots\times \mathbb{P}(\mathbb{V}_d)\Big).\) 
\begin{prop}\label{prop:A}
If $(n-1)d < m(2^{d}-1)$, then the hypergraph $\mathcal{G}'(T)$ has $\Omega (q^{dn - m})$ hyperedges,  $O(q^n)$ vertices and without $K^{(d)}_{2,\dots,  2}$.  In particular,  $\mathcal{G}'(T)$ achieves the lower bound in \eqref{eq:CPZ}.
\end{prop}
\begin{proof}
First $\mathcal{G}'(T)$ has the same number of vertices as $\mathcal{G}(T)$.
To see $\mathcal{G}'(T)$ is $K^{(d)}_{2,\ldots,  2}$,-free  suppose $\mathcal{G}(T)$ contains such a copy, and for each part $1 \le j \le d$ let its two vertices be
$[p_j]$ and $[q_j]$ with $p_j,q_j\in \mathbb{F}_q^{n+1}\setminus\{0\}$. Set $\mathbb{V}_j\coloneqq \langle p_j,q_j\rangle$. Since all $2^d$ edges of this $K^{(d)}_{2,\dots,2}$ belong to $\mathcal{G}(T)$, we have
$T(x_1,\dots,x_d)=0$ for every choice $x_j\in\{p_j,q_j\}$, which implies  $(\mathbb{V}_1,\dots,\mathbb{V}_d)\in D_T$.
Thus, every edge of the copy lies in
$\mathbb{P}(\mathbb{V}_1)\times\cdots\times \mathbb{P}(\mathbb{V}_d)$ and is deleted in $\mathcal{G}'(T)$. Moreover, Lemma~\ref{lem:VT} implies that $\mathcal{G}'(T)$ has $O(q^n)$ vertices, and  $\mathcal{G}'(T)$ is obtained from $\mathcal{G}(T)$ by deleting $(q+1)^d |D_T|$ hyperedges. Lemmas~\ref{lem:count_bad_configuration} and \ref{lem:VT} imply that the number of hyperedges of $\mathcal{G}'(T)$ is at least $(1 - (1 + o(1))/q ) |E_T| = \Omega(q^{dn - m})$.  
\end{proof}

We now apply Theorem~\ref{main theorem-2} to prove Corollary~\ref{cor:Erdox}, relying on the following lemma.

\begin{lemma}\label{effective prime decomposition}
There is a function $M: \mathbb{N} \times \mathbb{N} \times \mathbb{N} \to \mathbb{N}$ with the following property.  Given $f_1,\dots,  f_n \in \mathbb{F}_q[x_1,\dots,  x_N]$ of degrees at most $D$, every irreducible component of $X = \{a \in \overline{\mathbb{F}}_q: f_1(a) = \cdots = f_n(a) = 0\}$ is defined by some polynomials $g_1,\dots,  g_{\ell}\in \mathbb{F}_{q^r}[x_1,\dots,  x_N]$ for any $r$ divisible by $M(n,N,D)$.  Moreover,  we have $\max\{\ell,  \deg g_1,\dots,  \deg g_{\ell} \} \le M(n,N,D)$.
\end{lemma}
\begin{proof}
Let $X_1,\dots,  X_s$ be irreducible components of $X$.  According to the algorithm in \cite{chistov86} for the decomposition $X = \cup_{i=1}^s X_i$,  each $X_i$ is defined over $\mathbb{F}_{q^{M_i}}$ for some positive integer $M_i$.  Here for each $1 \le i \le s$,  $M_i$ only depends on $n,N,D$. Furthermore,  $X_i$ can be defined by at most $M_i$ equations over $\mathbb{F}_{q^{M_i}}$, each of degree at most $M_i$.  Lastly,  we observe that $s \le \deg X \le D^n$. This implies $M(n,N,D) \coloneqq \operatorname{lcm}(M_1,\dots,  M_s)$ is the desired function.
\end{proof}

\begin{proof}[Proof of Corollary~\ref{cor:Erdox}]
Given $T \in \Hom^d(\mathbb{F}^n,\mathbb{F}^m)$, we denote $X \coloneqq D_T(\overline{\mathbb{F}}_q)$. We claim that $\dim X \geq 2d(n-1)- m2^{d}$. Viewing $\Gr(2,\mathbb{F}_{q}^{n+1})$ as a subvariety of $\mathbb{P}\big( \bigwedge^2 \mathbb{F}_q^{n+1} \big)$ defined by quadratic polynomials,  we observe that $D_{T} \subseteq \Gr(2,\mathbb{F}_{q}^{n+1})^{d}$ is defined by $u$ polynomials in $v$ variables of degree most $w$,  for some $u,v, w$ depending only on $d$, $n$ and $m$. Let $M$ be the function in Lemma~\ref{effective prime decomposition}. Then for any positive integer $r$ divisible by some integer $r_0 \ge M(u,v,w)$, irreducible components of $X$ are defined over $\mathbb{F}_{q^r}$ by $s$ polynomials of degree at most $t$,  where $\max\{s,t\} \le M(u,v, w)$.  In particular,  $X$ is a variety defined over $\mathbb{F}_{q^r}$ whose geometrically irreducible components are also defined over $\mathbb{F}_{q^r}$.  We apply \cite[Theorem~8.1.1]{kedlaya_weil_cohomology} to $X$ and conclude that $|X(\mathbb{F}_{q^r})| = \Omega_{d,m,n}(q^{r(2d(n-1) - m2^d)})$.  By further increasing $r$,  if necessary,  we obtain $|X(\mathbb{F}_{q^r})| > 1/2 q^{r(2d (n-1) - m2^d)}$.  Hence we may define $C_1(n,d,m)$ to be the smallest such $r_0$. 

It remains to prove $\dim D_T \geq 2d(n-1)- m2^{d}$ for any $T \in \Hom^d(\mathbb{F}_q^{n+1},  \mathbb{F}_q^m)$.  Let $J_1$ be the subvariety of $\Gr(2,\overline{\mathbb{F}}_q^{n+1})^{d} \times \mathbb{P} ( \Hom^{d}( \overline{\mathbb{F}}_{q}^{n+1},\overline{\mathbb{F}}^{m}_{q}) )$ defined in \eqref{eq:J1} and let $\pi_2: J_1 \to \mathbb{P} ( \Hom^{d}(\overline{\mathbb{F}}_{q}^{n+1},\overline{\mathbb{F}}^{m}_{q}) )$ be the natural projection map.  Theorem~\ref{main theorem-2} implies that $\pi_2$ is surjective.  Note that $D_T =\pi_2^{-1}([T])$ for each $[T] \in \mathbb{P} ( \Hom^{d}(\overline{\mathbb{F}}_{q}^{n+1},\overline{\mathbb{F}}^{m}_{q}) )$.  Hence Lemmas~\ref{dimension of fibre},  \ref{upper-semicontinuous} and \ref{lem:dim J1} lead to
\begin{align*}
\dim D_T &\ge \dim J_1 - \dim \mathbb{P} ( \Hom^{d}( \overline{\mathbb{F}}_{q}^{n+1},\overline{\mathbb{F}}^{m}_{q}) ) \\
&= \big(2d(n-1) + m((n+1)^d - 2^d) - 1\big) - \big(m(n+1)^d - 1\big) \\&= 2d(n-1) - m 2^d. 
\end{align*}
\end{proof}
\section{ Python Code for $\mathrm{char}(\mathbb{F})=2$}\label{sec:appendB}
The following code uses \texttt{itertools} and \texttt{SymPy} to perform brute-force search and Gr\"obner basis computation over $\mathbb{F}_2$:
\begin{lstlisting}
import itertools
import sympy as sp

# Define the support of the alternating trilinear form T over F_2
triples = list(itertools.combinations(range(7), 3))
ones_T = {(1, 2, 5), (1, 3, 6), (1, 4, 7), (2, 3, 4), (5, 6, 7)}
form = [1 if tuple(i + 1 for i in t) in ones_T else 0 for t in triples]

def chart_matrix(pivots, k=5, n=7):
    nonp = [j for j in range(n) if j not in pivots]
    M = [[0] * n for _ in range(k)]
    vars_ = []
    for i, p in enumerate(pivots):
        M[i][p] = 1
    for i, p in enumerate(pivots):
        for j in nonp:
            if j > p:
                x = sp.symbols(f"x_{i}_{j}")
                M[i][j] = x
                vars_.append(x)
    return sp.Matrix(M), vars_

def equations_for_chart(B, vars_):
    equations = []
    for rows in itertools.combinations(range(5), 3):
        value = 0
        R = B[list(rows), :]
        for coeff, Tidx in zip(form, triples):
            if coeff:
                value += R[:, list(Tidx)].det()
        if vars_:
            poly = sp.Poly(value, *vars_, modulus=2).as_expr()
        else:
            poly = int(value) % 2
        equations.append(poly)
    return [f for f in equations if f != 0]

def eval_mod2(expr, sub):
    val = sp.expand(expr.subs(sub))
    return int(val) % 2

# Main loop over all 5-element pivot subsets of {0,1,...,6}
checked = 0
unit_by_constant = []
unit_by_groebner = []
failures = []
examples = []

for pivots in itertools.combinations(range(7), 5):
    B, vars_ = chart_matrix(pivots)
    equations = equations_for_chart(B, vars_)
    checked += 1
    
    if any(f == 1 for f in equations):
        unit_by_constant.append(pivots)
        continue
    
    found_point = None
    for vals in itertools.product([0, 1], repeat=len(vars_)):
        sub = dict(zip(vars_, vals))
        ok = True
        for f in equations:
            if eval_mod2(f, sub) != 0:
                ok = False
                break
        if ok:
            found_point = sub
            break
    if found_point is not None:
        examples.append((pivots, B.subs(found_point), found_point))
        failures.append((pivots, "has F2-point"))
        continue
    
    if not vars_:
        failures.append((pivots, "no free variables"))
        continue
    
    # Gr\"{o}bner basis over F_2
    G = sp.groebner(equations, *vars_, modulus=2, order="grevlex")
    if len(G.polys) == 1 and G.polys[0] == sp.Poly(1, *vars_, modulus=2):
        unit_by_groebner.append(pivots)
    else:
        failures.append((pivots, G))

# Output statistics
print("Total charts checked:", checked)
print("Empty by constant 1=0:", len(unit_by_constant))
print("Empty by Gr\"{o}bner basis:", len(unit_by_groebner))
print("Remaining failures:", len(failures))

if examples:
    print("\nFound F2 isotropic subspace example:")
    piv, mat, sol = examples[0]
    print("Pivot indices:", piv)
    print("Matrix:\n", mat)
else:
    if failures:
        print("\nFirst failure case:")
        print(failures[0][0], failures[0][1])
    else:
        print("\nAll charts are empty over F2 and its extensions.")
\end{lstlisting}
\end{appendices}

\end{document}